\documentclass[15pt]{article}

\usepackage{amsmath,amsfonts,amssymb,amsthm,amscd}
\numberwithin{equation}{section}

\newcounter{stepctr}
{\end{list}}

\newtheorem{thm}{Theorem}[section]

\newtheorem{cor}[thm]{Corollary}

\theoremstyle{definition}
\newtheorem{dfn}[thm]{Definition}
\newtheorem{ex}[thm]{Example}
\newtheorem{exs}[thm]{Examples}
\newtheorem{rema}[thm]{Remark}

\newtheorem{prob*}{Open problem}

\newcommand{\demo}{\begin{proof}}

\newcommand{\N}{\mathbb{N}}

\newcommand{\C}{\mathbb{C}}

\newcommand{\nullset}{\emptyset}

\newcommand{\cit}{\mathbb{C}}

\def\ll^2{{\mathcal L}(\ell^2(\N))}

\title
{\bf Abstract   Weyl-type  Theorems  }

\author{ M. Berkani}

\date{ }
\begin{document}

\maketitle \thispagestyle{empty}

\begin{abstract}\noindent\baselineskip=15pt
In this paper, we give a new approach for the study of Weyl-type theorems. Precisely we introduce the concepts of spectral valued   and spectral partitioning  functions. Using two  natural order relations on the set of spectral valued functions,  we reduce the question of relationship between Weyl-type theorems to the study of the set difference between the parts of the spectrum that are  involved. This study solves completely the question of relationship between  two spectral valued functions, comparable for one or the other order relation. Then several known results  about Weyl-type theorems becomes  corollaries of the results obtained.

\end{abstract}
 \baselineskip=15pt
 \footnotetext{\small \noindent  2010 AMS subject
classification: Primary 47A53, 47A10, 47A11 \\
\noindent Keywords: Spectral valued function, Partitioning, Spectrum, Weyl-type theorem. } \baselineskip=15pt

\section{Introduction}

Let $X$ be a Banach space, and let  $L(X)$ be the Banach algebra
 of all bounded linear operators acting on $X.$
 For $T\in L(X),$ we will denote by $N(T)$ the null space of $T$, by $\alpha(T)$ the nullity of $T$,
by $R(T)$ the range of $T$, by $\beta(T)$ its defect and by $T^*$ the adjoint of $T.$ We will denote also by $\sigma(T)$
the spectrum of $T$ and  by $\sigma_a(T)$ the approximate point spectrum of $T.$
If the range $R(T)$ of $T$ is closed and
$\alpha(T)<\infty$ (resp. $\beta(T)<\infty),$ then $T$ is called an upper
semi-Fredholm (resp. a lower semi-Fredholm) operator. If $T\in L(X)$ is either upper or
lower semi-Fredholm, then $T$ is called a semi-Fredholm operator, and the index of $T$ is defined by
$\mbox{ind}(T)=\alpha(T)-\beta(T)$. If both of $\alpha(T)$ and $\beta(T)$ are finite, then $T$ is called a Fredholm
operator. An operator $T\in L(X)$ is called a Weyl operator if it is a Fredholm operator of index zero. The Weyl spectrum $\sigma_{W}(T)$ of $T$ is
defined by $\sigma_{W}(T)= \{ \lambda \in \mathbb{C}\mid T-\lambda I$ is  not a Weyl operator\}.

For a bounded linear operator $T$ and a nonnegative integer $n,$ define $T_{[n]}$
to be the restriction of $T$ to $R(T^n),$ viewed as
 a  map from $R(T^n)$ into $R(T^n)$ (in particular $T_{[0]}=T$). If for some integer
 $n,$ the range space $R(T^n)$ is closed and $T_{[n]}$ is an upper (resp.
 a lower) semi-Fredholm operator, then $T$ is called an upper (resp. a lower)
 semi-B-Fredholm operator. A semi-B- Fredholm operator $T$  is an upper or a lower
 semi-B-Fredholm operator, and in  this case the index of
 $T$ is defined as the index of the semi-Fredholm operator $T_{[n]},$ see \cite{BS}. Moreover, if
 $T_{[n]}$ is a Fredholm operator, then $T$ is called a B-Fredholm operator, see \cite{BE1}.
 An operator $T\in L(X)$ is said to be a B-Weyl
 operator \cite {BW}, if it is a B-Fredholm operator of index zero. The B-Weyl spectrum $\sigma_{BW}(T)$ of $T$ is
defined by $\sigma_{BW}(T)= \{ \lambda \in \mathbb{C}\mid T-\lambda I$ is  not a B-Weyl operator\}.

The  ascent  $a(T)$ of an operator $T$ is defined by
 $a(T)=\mbox{inf} \{ n\in \mathbb{N} \mid N(T^n)=N(T^{n+1})\},$
and the  descent   $ \delta(T)$ of $T$, is defined by
 $\delta(T)= \mbox{inf} \{ n \in \mathbb{N} \mid R(T^n)= R(T^{n+1})\},$ with $ \mbox{inf}\, \emptyset= \infty.$

  According to \cite {H}, a complex number $ \lambda$ is a pole of the
resolvent of $T$ if and only if \,\,$ 0< $ max $(a(T- \lambda I),
\delta(T- \lambda I))< \infty.$ Moreover, if this is true, then $a(T-
\lambda I)= \delta(T- \lambda I).$ An operator $T$ is called Drazin invertible if $0$ is a pole of $T.$ The Drazin spectrum $\sigma_{D}(T)$ of $T$ is defined by $\sigma_{D}(T)=\{\lambda\in\C : T-\lambda I \, \,\text{is not  Drazin invertible}\}.$

 Define also the set  $LD(X)$ by $LD(X)=\{T\in L(X) : a(T)<\infty
\mbox{ and } R(T^{a(T)+1}) \mbox{ is closed}\}$ and
 $\sigma_{LD}(T)=\{\lambda\in\C : T-\lambda I\not\in LD(X)\}$ be the left Drazin spectrum. Following  \cite{BK},  an
 operator $T\in L(X)$ is said to be left Drazin invertible if $T\in
 LD(X)$. We say that $\lambda\in \sigma_a(T)$ is a left pole of $T$
if $T-\lambda I\in LD(X)$, and that $\lambda\in\sigma_a(T)$ is a
left pole of $T$ of finite rank if $\lambda$ is a left pole of $T$
and $\alpha(T-\lambda I)<\infty$.

Let $SF_{+}(X)$ be the class of all upper semi-Fredholm
operators and  $ SF_{+}^{-}(X)=\{T\in SF_{+}(X) : \mbox{ind}(T)\leq
0\}.$ The upper semi-Weyl  spectrum
$\sigma_{SF_+^-}(T)$ of $T$ is defined by
$\sigma_{SF_{+}^{-}}(T)=\{\lambda\in\cit : T-\lambda I \not\in
SF_{+}^{-}(X)\}.$  Similarly is defined the upper semi-B-Weyl spectrum
$\sigma_{SBF_+^-}(T)$ of $T.$

 An operator $T\in L(X)$ is called  upper semi-Browder if
it is upper semi-Fredholm operator of finite ascent, and is called
 Browder if it is a  Fredholm operator of finite ascent and descent.
The upper semi-Browder spectrum $\sigma_{uB}(T)$ of $T$ is
defined by $\sigma_{uB}(T)=\{\lambda\in \C : T-\lambda I \mbox { is
not upper semi-Browder}\}$, and the   Browder spectrum $\sigma_{B}(T)$ of $T$ is defined by
$\sigma_B(T)=\{\lambda\in \C : T-\lambda I \mbox { is not
Browder}\}$.

\vspace{1mm}
Below, we give a list of  symbols and notations we will use:
\vspace{1mm}

\noindent $E(T):$    eigenvalues of $T$ that are isolated in the spectrum $\sigma(T)$ of $T ,$\\
$E^0(T):$  eigenvalues of $T$ of finite multiplicity that are isolated in the spectrum
$\sigma(T)$ \, of $T,$\\
$E_a(T):$  eigenvalues of $T$ that  are isolated in the approximate point spectrum $\sigma_a(T)$  of  $T,$\\
$E_a^0(T):$  eigenvalues of $T$ of finite multiplicity that are isolated in the spectrum $\sigma_a(T)$  of  $T,$\\
$\Pi(T):$  poles of $T,$\\
$\Pi^0(T):$  poles of $T$ of finite rank,\\
$\Pi_a(T):$ left poles of $T,$\\
$\Pi_a^0(T):$ left poles of $T$ of finite rank,\\
$\sigma_{B}(T):$  Browder spectrum of $T,$ \\
$\sigma_{D}(T):$  Drazin spectrum of $T,$ \\
$\sigma_{LD}(T)$: left Drazin spectrum of $T.$\\
$\sigma_{uB}(T):$ upper semi-Browder spectrum of $T,$ \\
$\sigma_{BW}(T):$ B-Weyl spectrum of $T,$ \\
$\sigma_{W}(T):$  Weyl spectrum of $T,$ \\
$\sigma_{SF_+^-}(T):$ upper semi-Weyl  spectrum of $T,$ \\
$\sigma_{SBF_+^-}(T):$ upper semi-B-Weyl spectrum of $T,$ \\

Hereafter, the symbol $\bigsqcup$ stands for disjoint union, while $iso(A)$ and  $acc(A)$ means respectively isolated points and accumulation points of a given subset $A$ of $\mathbb{C}.$

 After the first step of this  introduction, we define in  the second  section of this paper the concepts of spectral valued functions and spectral partitioning functions. They are functions defined on the Banach algebra $L(X)$ and valued into $ \mathcal{P}(\mathbb{C})\times \mathcal{P}(\mathbb{C}),$ where  $\mathcal{P}(\mathbb{C})$ is the set of the  subsets of $\mathbb{C}.$   A spectral valued function $\Phi=(\Phi_1, \Phi_2)$ is a spectral partitioning, respectively a spectral a-partitioning,  valued function for an operator $T \in L(X)$ if $ \sigma(T)= \Phi_1(T)\bigsqcup \Phi_2(T), $  respectively if $ \sigma_a(T)= \Phi_1(T)\bigsqcup \Phi_2(T).$  Recall  that from \cite{WL}, if $T$ is a  normal operator acting on a Hilbert space, then $ \sigma(T)= \sigma_W(T) \bigsqcup E(T).$  Thus  a spectral valued function $\Phi= (\Phi_1, \Phi_2)$ could be  considered as  an "Abstract Weyl-type theorem," and an  operator $ T \in L(X)$ satisfies the abstract Weyl-type theorem $\Phi,$ if $\Phi$ is a spectral  partitioning or a-partitioning  function for $T.$\\
  Our main goal here is the study of abstract Weyl-type theorems and their relationship. By the study of relationship between two given abstract Weyl-type theorems  $\Phi$ and $\Psi$ we mean the answer of the  following question:  If an operator $T \in L(X)$ satisfies one of the two abstract Weyl-type theorems  $\Phi$ and $\Psi$, does $T$  satisfies the other one? The two abstract Weyl-type theorems  $\Phi$ and $\Psi$ are said to be equivalent if $T \in L(X)$ satisfies one of the two abstract Weyl-type theorems  $\Phi$ and $\Psi$   if and only $T$ satisfies the other one.
 To study the relationship between abstract Weyl-type theorems, we introduce  two  order relations  $\leq $  and $<<$  on the set  of  spectral valued functions. Then  the question of relationship between two comparable  spectral valued functions for the order $ \leq$  is solved in terms of set difference between  parts of the spectrum that are involved.   In the third section, following the same steps as in the second section,  we consider spectral a-partitioning functions  and we obtain similar results to those of the second section.

 In the forth section, we give some crossed results by considering  two spectral valued functions  comparable for the order $\leq$ , one partitioning  the spectrum  and the other  one partitioning  the approximate point spectrum. We obtain new kind of results, where the set difference $\sigma(T) \setminus \sigma_a(T)$ plays a crucial role. At the end of this section, we study the case of two  comparable  spectral  valued  functions for the order relation $<<$, and we answer in Theorem \ref{thm45} and Theorem \ref{thm46} the question of relationship between the two spectral valued functions.  \\
Globally, This study solves completely the question of relationship between two comparable spectral valued functions,  and several known results  about Weyl-type theorems appearing in recent literature  becomes  corollaries of the results obtained. To illustrate this,  we will give through the different sections, several examples as an application of the results obtained,  linking them to original references where they have been first  established.

As mentioned before, the original idea leading to a partition of the spectrum goes back to the famous paper by H. Weyl \cite{WL}. More recently, several authors worldwide had worked in this direction,  see for example \cite{AP}, \cite{BK}, \cite{CH}, \cite{DD}, \cite{DH}, \cite{DU},  \cite{RA} and \cite{XH}.

\vspace{5pt}

\section{ Partitioning  functions for the spectrum}

In this section we study  the relationship between  two  comparable  spectral valued  functions, when one of them is spectral partitioning and the other one would be also   spectral partitioning.

\begin{dfn} A spectral valued  function  is a function   $\Phi= (\Phi_1, \Phi_2) : L(X) \rightarrow  \mathcal{P}(\mathbb{C})\times \mathcal{P}(\mathbb{C})$ such that $ \,\, \forall \,\,  T \in L(X),  \Phi(T) \subset \sigma(T) \times  \sigma(T),$ where  $\mathcal{P}(\mathbb{C})$ is the set of the subsets of $\mathbb{C}.$

\end{dfn}

\begin{dfn} Let  $\Phi$ be a spectral valued  function.  We will  say that  $\Phi= (\Phi_1, \Phi_2)$ is a  spectral partitioning  function for an operator $T \in L(X),$ if    $\sigma(T)= \Phi_1(T) \bigsqcup \Phi_2(T).$
\end{dfn}

 A spectral valued function  $\Phi= (\Phi_1, \Phi_2)$ could be  considered as  an "Abstract Weyl-type theorem." An operator $ T \in L(X)$ satisfies the abstract Weyl-type theorem $\Phi$ if $\Phi$ is a spectral  partitioning  function for $T.$

\begin{ex}

\begin{itemize}
\item Let $\Phi_W(T)= (\sigma_W(T), E^0(T)), \,\, \forall \,\,T \in L(X).$ From \cite{WL}, it follows that $\Phi_W$ is a partitioning function for each normal operator acting on a Hilbert space.
\item Let $\Phi_{BW}(T)= (\sigma_{BW}(T), E(T)),\,\,  \forall \,\,T \in L(X).$ From \cite{BI}, it follows that $\Phi_{BW}$ is a partitioning function for each normal operator acting on a Hilbert space.

\end{itemize}

\end{ex}

\begin{dfn}

Let $\Psi$ and $\Phi$ be two spectral valued  functions.  We will say that $\Psi  \leq \Phi,$ if $\, \forall \, T \in L(X),$ we have
$\Phi_1(T) \subset \Psi_1(T)$ and $\Psi_2(T) \subset \Phi_2(T).$ We will say that  $\Psi << \Phi,$ if  \,\, $ \forall \, T \in L(X),$ we have $\Psi_1(T) \subset \Phi_1(T)$ and $\Psi_2(T) \subset \Phi_2(T).$

\end{dfn}

 It's easily seen that both   $\leq$ and $<<$  are order relations  on the set of spectral valued functions.

\begin{thm} \label{thm21} Let $T \in L(X)$ and let $\Phi$ be a spectral partitioning function for $T.$ If $\Psi$ is a spectral valued function such that $\Psi \leq \Phi,$   then $\Psi$ is  a spectral partitioning function for  $T$ if and only if $ \Psi_1(T) \setminus \Phi_1(T) = \Phi_2(T) \setminus \Psi_2(T).$
\end{thm}

\begin{proof}

Assume that $\Psi$ is a spectral partitioning function for  $T,$ then $\sigma(T)= \Psi_1(T) \bigsqcup \Psi_2(T)= \Phi_1(T) \bigsqcup \Phi_2(T).$
Hence $ \Psi_1(T) \setminus \Phi_1(T) = \Phi_2(T) \setminus \Psi_2(T).$\\
 Conversely assume that $ \Psi_1(T) \setminus \Phi_1(T) = \Phi_2(T) \setminus \Psi_2(T).$ Since $\Phi$ is a spectral partitioning function for $T,$ then $\sigma(T)= \Phi_1(T) \bigsqcup \Phi_2(T).$ As $\Psi \leq \Phi,$ then
$\Phi_1(T) \subset \Psi_1(T)$ and so
$\sigma(T)= \Psi_1(T) \cup \Phi_2(T)= \Psi_1(T) \cup (\Phi_2(T) \setminus \Psi_2(T)) \cup \Psi_2(T) .$  As  $ \Psi_1(T) \setminus \Phi_1(T) = \Phi_2(T) \setminus \Psi_2(T),$ then $ \Phi_2(T) \setminus \Psi_2(T) \subset \Psi_1(T).$ Hence $\sigma(T) \subset  \Psi_1(T) \cup \Psi_2(T).$ As we have always $ \Psi_1(T) \cup \Psi_2(T) \subset \sigma(T),$ then $\sigma(T) =  \Psi_1(T) \cup \Psi_2(T).$   Moreover we have   $ \Psi_1(T) \cap \Psi_2(T) = [\Phi_1(T)\cup (\Psi_1(T)\setminus \Phi_1(T))]  \cap \Psi_2(T)= [\Phi_1(T) \cap \Psi_2(T)] \cup [(\Psi_1(T)\setminus \Phi_1(T)) \cap \Psi_2(T)]= [\Phi_1(T) \cap \Psi_2(T)] \cup [(\Phi_2(T)\setminus \Psi_2(T)) \cap \Psi_2(T)]= \emptyset.$  Hence
$\sigma(T)= \Psi_1(T) \bigsqcup \Psi_2(T) $ and $\Psi$ is a spectral partitioning function for $T.$ \end{proof}

In the following corollary, as an application of Theorem \ref{thm21},  we give a direct proof of \cite [Theorem 3.9]{BK}

\begin{cor} \label{cor21}If $\Phi_{BW}$ is a spectral  partitioning function for $T,$ then $\Phi_{W}$ is also a spectral partitioning function for $T.$ \end{cor}
\begin{proof} Observe first that $\Phi_{W} \leq \Phi_{BW}.$ Then  if $\Phi_{BW}$ is a spectral partitioning function for $T,$ it's easily seen that  $\sigma_{W}(T) \setminus \sigma_{BW}(T)= E(T) \setminus E^0(T).$ From Theorem \ref{thm21}, it follows that $\Phi_{W}$ is also a spectral  partitioning function for $T.$  \end{proof}
Similarly to Theorem \ref{thm21}, we have the following theorem, which we give without proof.

\begin{thm} \label{thm22} Let $T \in L(X)$ and let $\Psi$ be a spectral partitioning function for $T.$ If $\Phi$ is a spectral valued function such that $\Psi \leq \Phi,$   then $\Phi$ is a spectral partitioning function for  $T$ if and only if $ \Psi_1(T) \setminus \Phi_1(T) = \Phi_2(T) \setminus \Psi_2(T).$
\end{thm}

\begin{rema}\label{rema21} \cite[Example 3.12]{BK} There exist operators $T \in L(X)$ such that $\Phi_W$ is a spectral  partitioning function for $T$ but $\Phi_{BW}$ is not  a spectral partitioning function for $T.$
Indeed, let us consider the operator  $Q$ be defined for each $x=(\xi_i)\in l^1$ by
\begin{equation*}
Q(\xi_1,\xi_2,\xi_3,\dots, \xi_k,\dots)
= ( 0, \alpha_1\xi_1, \alpha_2\xi_2,\dots, \alpha_{k-1}\xi_{k-1},\dots),
\end{equation*}
where $(\alpha_i)$ is a sequence of complex numbers such that
$0<|\alpha_i|\le 1$ and $\sum_{i=1}^\infty |\alpha_i|<\infty$.  We
observe that
\begin{equation*}
\overline{R(Q^n)} \ne R(Q^n), \quad n=1,2,\dots
\end{equation*}
Indeed, for a given $n\in\N$ let $x^{(n)}_k = (1,\dots ,
1,0,0,\dots)$ (with $n+k$ times $1$).  Then the limit
$y^{(n)}=\lim_{k\to\infty}Q^n x^{(n)}_k$ exists and lies in
$\overline{R(Q^n)}$.  However, there is no element
$x^{(n)}\in\ell^1$ satisfying the equation $Q^nx^{(n)}=y^{(n)}$ as the
algebraic solution to this equation is  $(1,1,1,\dots
)\notin\ell^1$.

Define $T$ on $X=\ell^1\oplus \ell^1$ by
$T=Q\oplus 0$.  Then $N(T)=\{0\}\oplus \ell^1$,
$\sigma(T)=\{0\}$,
$E(T)=\{0\}$,  $E_0(T)=\nullset$.  Since $R(T^n) = R(Q^n) \oplus
\{0\}$, $R(T^n)$ is not closed for any $n\in\N$; so $T$ is not a
$B$-Weyl operator, and $\sigma_{BW}(T)=\{0\}$.  Further, $T$ is not a Fredholm operator and $\sigma_{W}(T) = \{0\}$.

Hence  $\Phi_W$ is a spectral  partitioning function for $T$ but $\Phi_{BW}$ is not  a spectral partitioning function for $T.$

\end{rema}

\begin{dfn} The Drazin spectral valued function $\Phi_D$  and the Browder spectral valued function $\Phi_B$ are defined respectively by:  $$\Phi_D(T)= ( \sigma_{BW}(T), \Pi(T)), \, \, \Phi_B(T)= ( \sigma_{W}(T), \Pi^0(T)),  \forall T \in L(X).$$

\end{dfn}

\begin{thm}\label{thm23} Let $T \in L(X).$ Then the  Drazin spectral valued function $\Phi_D$ is a spectral partitioning   function for $T$ if and only if the    Browder spectral  valued function $\Phi_B$ is a spectral partitioning function for $T.$

\end{thm}

\begin{proof} Observe first that $\Phi_B \leq \Phi_D.$ If $\Phi_D$ is a spectral partitioning   function for $T,$ then
$ \sigma_W(T) \setminus \sigma_{BW}(T)= \Pi(T) \setminus \Pi^0(T).$  From Theorem \ref{thm21}, we conclude that $\Phi_B$ is a spectral partitioning function for $T.$
Conversely assume that  $\Phi_B$ is a spectral  partitioning function for $T.$ Let us show that $ \sigma_W(T) \setminus \sigma_{BW}(T)= \Pi(T) \setminus \Pi^0(T).$ The inclusion $ \Pi(T) \setminus \Pi^0(T) \subset \sigma_W(T) \setminus \sigma_{BW}(T)$ is obvious. For the reverse inclusion, let $ \lambda \in \sigma_W(T) \setminus \sigma_{BW}(T).$ Then from \cite[Corollary 3.2]{BS}, $\lambda $ is isolated in $\sigma_W(T).$  As $\Phi_B$ is a spectral partitioning   function for $T,$ then $\sigma(T)= \sigma_W(T) \bigsqcup \Pi^0(T).$ Hence $\lambda$ is isolated in $\sigma(T).$ As $ \lambda \notin \sigma_{BW}(T),$  then from \cite [Theorem 2.3]{BW},  $\lambda \in \Pi(T) \setminus \Pi_0(T).$ From Theorem \ref{thm22}, it follows that $\Phi_D$ is a spectral partitioning function for $T.$
\end{proof}

The direct implication of Theorem \ref{thm23} had been proved in \cite[Theorem 3.15]{BK},  while the reverse implication was posed as a question in \cite[p. 374]{BK},   and answered in \cite[Theorem 3.1]{CM} and \cite[Theorem 2.1]{AZ}.

\begin{rema} If  the  Drazin spectral valued function $\Phi_D$ is a spectral partitioning   function for $T \in L(X),$ then  $\sigma_{BW}(T)= \sigma_D(T)$ and $\sigma_{W}(T)= \sigma_B(T).$
\end{rema}

\begin{exs}
The following table summarize some of  spectral valued functions considered recently as partitioning functions.
\begin{center}
\vbox{
\[
\begin{tabular}{|l|l|}
\hline
\multicolumn{2}{|c|} {\textbf Spectral valued functions-1} \\[5pt] \hline

$\Phi_W(T)= (\sigma_{W}(T),E^0(T))$ & $\Phi_B(T)=(\sigma_{W}(T),\Pi^0(T))$\\[5pt]
$\Phi_{gW}(T)=(\sigma_{BW}(T),E(T))$& $\Phi_{gB}(T)= (\sigma_{BW}(T),\Pi(T))$ \\[5pt]
$\Phi_{Bw}=(\sigma_{BW}(T), E^0(T))$ & $\Phi_{Bb}(T)=(\sigma_{BW}(T),\Pi^0(T))$ \\[5pt]
$\Phi_{aw}(T)=(\sigma_{W}(T),E_a^0(T))$& $\Phi_{ab}(T)=(\sigma_{W}(T),\Pi_a^0(T))$ \\[5pt]
$\Phi_{gaw}(T)=(\sigma_{BW}(T),E_a(T))$& $\Phi_{gab}(T)=(\sigma_{BW}(T),\Pi_a(T))$ \\[5pt]
$ \Phi_{Baw}(T)=(\sigma_{BW}(T),E_a^0(T))$ & $\Phi_{Bab}(T)=(\sigma_{BW}(T),\Pi_a^0(T))$\\[5pt]

\hline

\end{tabular}
\]

\begin{center}
\underline{Table~1}
\end{center}}

\end{center}
Among the spectral valued functions listed in Table~1, we consider the following cases to illustrate  the use of Theorem \ref{thm21} and Theorem \ref{thm22}

\begin{itemize}

\item It is shown in \cite[Theorem 3.5 ]{BZ3} that if $\Phi_{gaw}$ is  a spectral partitioning function for $ T \in L(X)$, then  $\Phi_{gab}$  is also a partitioning function for $T.$ As  $\Phi_{gab} \leq \Phi_{gaw}$,\, to prove this result  using  Theorem \ref{thm21},   it is enough to prove that $\emptyset= \sigma_{BW}(T) \setminus \sigma_{BW}(T)=  E_a(T) \setminus  \Pi_a(T),$ which is  the case from \cite[Theorem 2.8]{BK}.

\item It is shown in \cite[Theorem 2.9 ]{B1} that if $\Phi_{W}$ is a spectral partitioning function for $ T \in L(X)$, then  $\Phi_{gW}$  is  a partitioning function for $T$  if and only if $E(T)= \Pi(T).$  As  $\Phi_{W} \leq \Phi_{gW}$,\,
    to prove this result using  Theorem \ref{thm22},  it is enough to prove that $\emptyset= \sigma_{W}(T) \setminus \sigma_{W}(T)=  E(T) \setminus  \Pi(T),$ which is  the case from \cite[Corollary2.6]{BW}.

\item It is shown in \cite[Corollary 5 ]{BAR} that if $\Phi_{W}$ is a spectral partitioning function for $ T \in L(X)$, then  $\Phi_{B}$  is also a spectral partitioning function for $T.$ To see this using  Theorem \ref{thm21},  as  $\Phi_{B} \leq \Phi_{W}$,\, it is enough to prove that $\emptyset= \sigma_{W}(T) \setminus \sigma_{W}(T)=  E^0(T) \setminus  \Pi^0(T),$ which is  the case from \cite[Theorem 4.2]{BW}.

\end{itemize}

\end{exs}

\section{ Partitioning  functions for the approximate spectrum}

In this section we study  the relationship between  two  comparable  spectral valued  functions, when one of them is spectral a-partitioning and the other one would be also   spectral a-partitioning.

\begin{dfn} Let  $\Phi$ be a spectral valued  function and let $T \in L(X).$  We will  say that $\Phi$ is a spectral  a-partitioning  function for $T$  if    $\sigma_a(T)= \Phi_1(T) \bigsqcup \Phi_2(T).$

\end{dfn}

\begin{ex}

\begin{itemize}
\item Let $\Phi_{aW}(T)= (\sigma_{SF_+^-}(T), E_a^0(T)), \,\, \forall \,\,T \in L(X).$ From \cite{RA}, it follows that $\Phi_{aW}$ is a spectral a-partitioning function for each normal operator acting on a Hilbert space.
\item Let $\Phi_{gaW}(T)= (\sigma_{SBF_+^-}(T), E_a(T)),\,\,  \forall \,\,T \in L(X).$  In the case of a normal operator $T$ acting on a Hilbert space, we have  $ \sigma(T)= \sigma_a(T)$ and $\Phi_{gaW}(T)=\Phi_{gW}(T).$ From \cite[Theorem 4.5]{BI}, it follows that $\Phi_{gaW}$ is a spectral  a-partitioning for $T.$

\end{itemize}

\end{ex}

\begin{thm} \label{thm31} Let $T \in L(X)$ and let $\Phi$ be  a spectral a-partitioning function for $T.$ If $\Psi$ is a spectral valued function such that $\Psi \leq \Phi,$   then $\Psi$ is   a spectral a-partitioning function for  $T$ if and only if $ \Psi_1(T) \setminus \Phi_1(T) = \Phi_2(T) \setminus \Psi_2(T).$
\end{thm}

\begin{proof}
Assume that $\Psi$ is a spectral a-partitioning function for  $T,$ then $\sigma_a(T)= \Psi_1(T) \bigsqcup \Psi_2(T)= \Phi_1(T) \bigsqcup \Phi_2(T).$
Hence $ \Psi_1(T) \setminus \Phi_1(T) = \Phi_2(T) \setminus \Psi_2(T).$

  Conversely if $\Phi$ is  a spectral a-partitioning function for $T$ and  $ \Psi_1(T) \setminus \Phi_1(T) = \Phi_2(T) \setminus \Psi_2(T).$  Then $\Psi_1(T) \subset \sigma_a(T)$  and $\Psi_2(T) \subset \sigma_a(T).$ As  $\sigma_a(T)= \Phi_1(T) \bigsqcup \Phi_2(T)$  and $\Psi \leq \Phi,$ then
$\Phi_1(T) \subset \Psi_1(T)$ and so
$\sigma_a(T)= \Psi_1(T) \cup \Phi_2(T)= \Psi_1(T) \cup (\Phi_2(T) \setminus \Psi_2(T)) \cup \Psi_2(T) =\Psi_1(T) \cup (\Psi_1(T) \setminus \Phi_1(T)) \cup \Psi_2(T)= \Psi_1(T) \cup  \Psi_2(T).$ Moreover, we have   $ \Psi_1(T) \cap \Psi_2(T) = [\Phi_1(T)\cup (\Psi_1(T)\setminus \Phi_1(T))]  \cap \Psi_2(T)= [\Phi_1(T) \cap \Psi_2(T)] \cup [(\Psi_1(T)\setminus \Phi_1(T)) \cap \Psi_2(T)]= [\Phi_1(T) \cap \Psi_2(T)] \cup [(\Phi_2(T)\setminus \Psi_2(T)) \cap \Psi_2(T)]=\emptyset.  $ Hence
$\sigma_a(T)= \Psi_1(T) \bigsqcup \Psi_2(T) $ and $\Psi$ is  a spectral a-partitioning function for $T.$
\end{proof}

In the following corollary, as an application of Theorem \ref{thm31},  we give a direct proof of \cite [Theorem 3.11]{BK}

\begin{cor} If $\Phi_{gaW}$ is  a spectral a-partitioning function for $T,$ then $\Phi_{aW}$ is also a spectral  a-partitioning function for $T.$ \end{cor}

\begin{proof} If $\Phi_{gaW}$ is a spectral  a-partitioning function for $T,$ then it's easily seen that  $\sigma_{SF_+^-}(T) \setminus \sigma_{SBF_+^-}(T)= E_a(T) \setminus E_a^0(T).$ From Theorem \ref{thm31}, it follows that $\Phi_{aW}$ is also a spectral  a-partitioning function for $T.$ \end{proof}

Similarly to Theorem \ref{thm31}, we have the following theorem, which we give without proof.

\begin{thm} \label{thm32}  Let $T \in L(X)$ and let $\Psi$ be a spectral a-partitioning function for $T.$ If $\Phi$ is a spectral valued function such that $\Psi \leq \Phi,$   then $\Phi$ is  a spectral a-partitioning function for  $T$ if and only if $ \Psi_1(T) \setminus \Phi_1(T) = \Phi_2(T) \setminus \Psi_2(T).$
\end{thm}

\begin{rema} The spectral valued function  $\Phi_{aW}$  is a spectral a-partitioning function for the operator $T$ considered in Remark \ref{rema21},  but $\Phi_{gaW}$ is not  a spectral a-partitioning  function for $T.$
\end{rema}

\begin{dfn} The Left-Drazin spectral valued function $\Psi_{gaB}$ is defined by:
$$\Psi_{gaB}(T)= ( \sigma_{SBF_+^-}(T), \Pi_a(T)), \, \forall T \in L(X),$$
while the Upper-Browder spectral valued
 function $\Psi_{aB}$ is defined on $L(X)$ by:
 $$\Psi_{aB}(T)= ( \sigma_{SF_+^-}(T), \Pi_a^0(T)),  \forall T \in L(X).$$

\end{dfn}

\begin{thm} \label{thm33} Let $T \in L(X).$ Then the Left-Drazin spectral valued function $\Psi_{gaB}$ is a spectral a-partitioning function for $T$ if and only if the upper-Browder spectral valued function $\Psi_{aB}$ is a spectral a-partitioning function for $T.$

\end{thm}

\begin{proof} Observe first that $\Psi_{aB} \leq \Psi_{gaB}.$ If $\Psi_{gaB}$ is a spectral a-partitioning   function for $T,$ then
$ \sigma_{SF_+^-}(T) \setminus \sigma_{SBF_+^-}(T)= \Pi_a(T) \setminus \Pi_a^0(T).$  From Theorem \ref{thm31}, we conclude that $\Psi_{aB}$ is a spectral  a-partitioning function for $T.$
Conversely assume that  $\Psi_{aB}$ is a spectral a-partitioning function for $T.$ Let us show that $ \sigma_{SF_+^-}(T) \setminus \sigma_{SBF_+^-}(T)= \Pi_a(T) \setminus \Pi_a^0(T).$ The inclusion $ \Pi_a(T) \setminus \Pi_a^0(T) \subset \sigma_{SF_+^-}(T) \setminus \sigma_{SBF_+^-}(T)$ is obvious. For the reverse inclusion, let $ \lambda \in \sigma_{SF_+^-}(T) \setminus \sigma_{SBF_+^-}(T).$ From \cite[Corollary 3.2]{BS},  $\lambda $ is isolated in \, $\sigma_{SF_+^-}(T).$  As $\Psi_{aB}$ is a spectral partitioning   function for $T,$ then $\sigma_a(T)= \sigma_{SF_+^-}(T) \bigsqcup \Pi_a^0(T).$ Hence $\lambda$ is isolated in $\sigma_a(T).$ From \cite[Theorem 2.8]{BK},  it follows that  $\lambda \in \Pi_a(T) \setminus \Pi_a^0(T).$
\end{proof}

The direct implication of Theorem \ref{thm33} had been proved in \cite[Theorem 3.8]{BK},  while the reverse implication was posed as a question in \cite[p. 374]{BK},   and answered in \cite[Theorem 1.3]{XH} and \cite[Theorem 2.2]{AZ}.

\begin{rema} If  the  Left-Drazin spectral valued function $\Psi_{gaB}$ is a spectral a-partitioning   function for $T \in L(X),$ then  $\sigma_{SBF_-^+}(T)= \sigma_{LD}(T)$ and $\sigma_{SF_-^+}(T)= \sigma_{uB}(T).$
\end{rema}

\begin{exs}

The following table summarize some of spectral valued functions considered recently as a-partitioning functions.

\begin{center}

\vbox{
\[
\begin{tabular}{|l|l|}
\hline

\multicolumn{2}{|c|} {\textbf Spectral valued functions-2} \\[5pt] \hline
$\Psi_{aW}(T)=(\sigma_{SF_+^-}(T),E_a^0(T))$ & $\Psi_{aB}(T)=(\sigma_{SF_+^-}(T), \Pi_a^0(T))$\\[5pt]
$\Psi_{gaW}(T)=(\sigma_{SBF_+^-}(T),E_a(T))$& $\Psi_{gaB}(T)= (\sigma_{SBF_+^-}(T),\Pi_a(T))$ \\[5pt]
$\Psi_{w}(T)= (\sigma_{SF_+^-}(T),E^0(T))$ & $\Psi_{b}(T)= (\sigma_{SF_+^-}(T), \Pi^0(T))$ \\[5pt]
 $\Psi_{gw}(T)=(\sigma_{SBF_+^-}(T),E(T))$ & $\Psi_{gb}(T)=(\sigma_{SBF_+^-}(T),\Pi(T))$\\[5pt]
$ \Psi_{SBw}(T)= (\sigma_{SBF_+^-}(T),E^0(T))$&$\Psi_{SBb}(T) =(\sigma_{SBF_+^-}(T),\Pi^0(T))$ \\[5pt]
 $\Psi_{SBaw}(T)=(\sigma_{SBF_+^-}(T),E_a^0(T)) $ & $\Psi_{SBab}(T)=(\sigma_{SBF_+^-}(T), \Pi_a^0(T))$\\[5pt]

\hline
\end{tabular}
\]

\begin{center}
\underline{Table~2}
\end{center}}

\end{center}

Among the spectral valued functions listed in Table~2, we consider the following cases to illustrate  the use of Theorem \ref{thm31} and Theorem \ref{thm32}

\begin{itemize}

\item It is shown in \cite[Theorem 2.15]{BZ1} that if $\Psi_{gw}$ is a spectral a-partitioning function for $ T \in L(X)$, then  $\Psi_{gb}$  is also a spectral a-partitioning function for $T.$ Since $\Psi_{gb}  \leq  \Psi_{gw},$ to prove this result  using  Theorem \ref{thm31},   it is enough to prove that $\emptyset= \sigma_{SBF_+^-}(T) \setminus \sigma_{SBF_+^-}(T)=  E(T) \setminus  \Pi(T),$ which is  the case from
    \cite[Theorem 4.2]{BI}.

\item It is shown in \cite[Corollary 3.3 ]{BK} that if $\Psi_{gaW}$ is a spectral a-partitioning function for $ T \in L(X)$, then  $\Psi_{gaB}$  is   a spectral  a-partitioning function for $T.$ Since $\Psi_{gaB}  \leq  \Psi_{gaW},$  to prove this result  using  Theorem \ref{thm31},   it is enough to prove that $\emptyset= \sigma_{SBF_+^-}(T) \setminus \sigma_{SBF_+^-}(T)=  E_a(T) \setminus  \Pi_a(T),$ which is  the case from
    \cite[Theorem 2.8]{BK}.

\end{itemize}
\end{exs}

\section {Crossed  Results}
In this section we consider the situation of two  comparable spectral valued  functions, one is spectral partitioning, while the other one would be  spectral a-partitioning, and vice-versa.

\begin{thm} \label{thm41} Let $T \in L(X)$ and let $\Phi$ be  a spectral partitioning function for $T.$ If $\Psi$ is a spectral valued function such that $\Psi \leq \Phi,$   then $\Psi$ is a spectral a-partitioning function for  $T$ if and only if $ \Phi_2(T) \setminus \Psi_2(T) = (\Psi_1(T) \setminus \Phi_1(T)) \bigsqcup  (\sigma(T)\setminus \sigma_a(T)).$
\end{thm}

\begin{proof}
If $\Psi$ is a spectral partitioning function for  $T,$ then $\sigma(T)= [\Psi_1(T) \cup \Psi_2(T)] \bigsqcup [(\sigma(T) \setminus \sigma_a(T))]= [\Phi_1(T) \cup (\Psi_1(T) \setminus \Phi_1(T))\cup\Psi_2(T)]\bigsqcup[(\sigma(T) \setminus \sigma_a(T))]. $  Hence  $\Phi_2(T) \setminus \Psi_2(T)=(\Psi_1(T) \setminus \Phi_1(T)) \bigsqcup  (\sigma(T)\setminus \sigma_a(T)).$\\
Conversely assume that $\Phi$ is  a spectral partitioning function for $T$ and $ \Phi_2(T) \setminus \Psi_2(T) = (\Psi_1(T) \setminus \Phi_1(T)) \bigsqcup (\sigma(T)\setminus \sigma_a(T)).$ Then $ \sigma(T)= \Phi_1(T)\bigsqcup \Phi_2(T)= \Phi_1(T) \cup (\Phi_2(T)\setminus \Psi_2(T)) \cup \Psi_2(T)=
 \Phi_1(T) \cup(\Psi_1(T) \setminus \Phi_1(T)) \cup  (\sigma(T)\setminus \sigma_a(T))\cup \Psi_2(T)= \Psi_1(T)  \cup \Psi_2(T)\cup (\sigma(T)\setminus \sigma_a(T)).$  Hence $\sigma_a(T)\subset \Psi_1(T)  \cup \Psi_2(T).$

 Since  $ \Phi_2(T) \setminus \Psi_2(T) = (\Psi_1(T) \setminus \Phi_1(T)) \bigsqcup (\sigma(T)\setminus \sigma_a(T)),$ then $ \Psi_1(T)  \cup \Psi_2(T)\subset \sigma_a(T).$ Moreover as $\Psi_2(T) \subset  \Phi_2(T),$
    we have $\Psi_1(T) \cap \Psi_2(T)= (\Phi_1(T) \cup (\Psi_1(T) \setminus \Phi_1(T)) \cap \Psi_2(T)= \emptyset.$   Then
$\sigma_a(T)= \Psi_1(T)  \bigsqcup \Psi_2(T)$  and $\Psi$ is a spectral a-partitioning function for $T.$
\end{proof}

\begin{cor}\label{cor41} Let $T \in L(X)$ and let $\Psi$ be a spectral a-partitioning function for $T.$ If $\Phi$ is a spectral valued function such that $\Psi \leq \Phi,$   then $\Phi$ is a spectral partitioning function for  $T$ if and only if $ \Phi_2(T) \setminus \Psi_2(T) = (\Psi_1(T) \setminus \Phi_1(T)) \bigsqcup  (\sigma(T)\setminus \sigma_a(T)).$
\end{cor}

\begin{proof}
Assume that $\Phi$  is a spectral partitioning  function for $T.$ As  $\Psi \leq \Phi,$ and $\Psi$ is  a spectral a-partitioning function for $T,$ then from  Theorem \ref{thm41}, we have $ \Phi_2(T) \setminus \Psi_2(T) = (\Psi_1(T) \setminus \Phi_1(T)) \bigsqcup (\sigma(T)\setminus \sigma_a(T)).$\\
Conversely assume that $\Psi$ is  a spectral a-partitioning function for $T$ and $ \Phi_2(T) \setminus \Psi_2(T) = (\Psi_1(T) \setminus \Phi_1(T)) \bigsqcup (\sigma(T)\setminus \sigma_a(T)).$ Then $ \sigma(T)= [\Psi_1(T) \cup \Psi_2(T)]\bigsqcup [(\sigma(T)\setminus \sigma_a(T))]=   [ \Phi_1(T) \cup (\Psi_1(T) \setminus \Phi_1(T)) \cup \Psi_2(T)] \bigsqcup [\sigma(T)\setminus \sigma_a(T)]=
\Phi_1(T) \cup (\Phi_2(T) \setminus \Psi_2(T))\cup \Psi_2(T)= \Phi_1(T) \cup \Phi_2(T).$ \\
 As $\Phi_1(T) \subset \Psi_1(T),$ we have $\Phi_1(T) \cap \Phi_2(T)= \Phi_1(T) \cap [(\Phi_2(T) \setminus \Psi_2(T)) \cup \Psi_2(T)]= \Phi_1(T) \cap [(\Psi_1(T) \setminus \Phi_1(T)) \cup  (\sigma(T)\setminus \sigma_a(T)) \cup \Psi_2(T)]= \emptyset.$  Therefore
$\sigma(T)= \Phi_1(T)  \bigsqcup \Phi_2(T)$  and $\Phi$ is a spectral partitioning function for $T.$  \end{proof}

Similarly to Theorem  \ref {thm41}  and Corollary  \ref{cor41} we have the following two results.

\begin{thm} \label{thm42} Let $T \in L(X)$ and let $\Phi$ be a spectral a-partitioning  function for $T.$ If $\Psi$ is a spectral valued function such that $\Psi \leq \Phi,$   then $\Psi$ is a spectral partitioning function for  $T$ if and only if $ \Psi_1(T) \setminus \Phi_1(T) = (\Phi_2(T) \setminus \Psi_2(T)) \bigsqcup(\sigma(T)\setminus \sigma_a(T)).$
\end{thm}
\begin{proof}

If $\Psi$ is a spectral partitioning function for  $T,$ then $\sigma(T)= \Psi_1(T) \bigsqcup \Psi_2(T) = \Phi_1(T) \cup (\Psi_1(T) \setminus \Phi_1(T))\cup\Psi_2(T).$  Hence  $\Psi_1(T) \setminus \Phi_1(T)=(\Phi_2(T) \setminus \Psi_2(T)) \bigsqcup  (\sigma(T)\setminus \sigma_a(T)).$\\
Conversely assume that $\Phi$ is  a spectral a-partitioning function for $T$ and $ \Psi_1(T) \setminus \Phi_1(T) = (\Phi_2(T) \setminus \Psi_2(T)) \bigsqcup (\sigma(T)\setminus \sigma_a(T)).$ Then $ \sigma(T)= \Phi_1(T)\cup \Phi_2(T)\cup (\sigma(T)\setminus \sigma_a(T)) = \Phi_1(T) \cup (\Phi_2(T)\setminus \Psi_2(T)) \cup \Psi_2(T)  \cup  (\sigma(T)\setminus \sigma_a(T))= \Phi_1(T)  \cup (\Psi_1(T)\setminus \Phi_1(T)) \cup \Psi_2(T). $  Hence $\sigma(T) =\Psi_1(T)  \cup \Psi_2(T).$ Since  $ \Psi_1(T) \setminus \Phi_1(T) = (\Phi_2(T) \setminus \Psi_2(T)) \bigsqcup (\sigma(T)\setminus \sigma_a(T))$ and $\Psi_2(T) \subset \Phi_2(T),$ then $ \Psi_1(T)  \cap \Psi_2(T)= [\Phi_1(T) \cup (\Psi_1(T)\setminus \Phi_1(T)] \cap \Psi_2(T)= \emptyset. $ Hence  $\sigma(T)= \Psi_1(T)  \bigsqcup \Psi_2(T)$  and $\Psi$ is a spectral partitioning function for $T.$  \end{proof}

\begin{cor} \label{cor42} Let $T \in L(X)$ and let $\Psi$ be  a spectral  partitioning function for $T.$ If $\Phi$ is a spectral valued function such that $\Psi \leq \Phi,$   then $\Phi$ is a spectral a-partitioning function for  $T$ if and only if $ \Psi_1(T) \setminus \Phi_1(T) = (\Phi_2(T) \setminus \Psi_2(T)) \bigsqcup (\sigma(T)\setminus \sigma_a(T)).$
\end{cor}

\begin{proof}

Assume that $\Phi$  is a spectral a-partitioning  function for $T.$ As  $\Psi \leq \Phi,$ and $\Psi$ is  a spectral partitioning function for $T,$ then from  Theorem \ref{thm42}, we have $ \Psi_1(T) \setminus \Phi_1(T) = (\Phi_2(T) \setminus \Psi_2(T)) \bigsqcup (\sigma(T)\setminus \sigma_a(T)).$\\
Conversely assume that $\Psi$ is  a spectral partitioning function for $T$ and $ \Psi_1(T) \setminus \Phi_1(T) = (\Phi_2(T) \setminus \Psi_2(T)) \bigsqcup (\sigma(T)\setminus \sigma_a(T)).$ Then
$ \sigma(T)= \Psi_1(T)\cup \Psi_2(T) = \Phi_1(T) \cup (\Psi_1(T)\setminus \Phi_1(T)) \cup \Psi_2(T)=  \Phi_1(T) \cup  (\Phi_2(T) \setminus \Psi_2(T)) \cup (\sigma(T)\setminus \sigma_a(T)) \cup \Psi_2(T) = \Phi_1(T)  \cup  \Phi_2(T) \cup (\sigma(T)\setminus \sigma_a(T)). $

 Hence $\sigma_a(T) \subset \Phi_1(T)  \cup \Phi_2(T).$ Since  $ \Psi_1(T) \setminus \Phi_1(T) = (\Phi_2(T) \setminus \Psi_2(T)) \bigsqcup (\sigma(T)\setminus \sigma_a(T))$ and $\Psi_2(T) \cap \Psi_1(T)= \emptyset,$ then $ \Phi_1(T)  \cup \Phi_2(T) \subset \sigma_a(T).$ Moreover, we have  $\Phi_1(T) \cap \Phi_2(T)= \Phi_1(T) \cap [(\Phi_2(T) \setminus \Psi_2(T)) \cup \Psi_2(T)] $ and since $\Phi_1(T) \subset \Psi_1(T),$  then  $\Phi_1(T) \cap \Phi_2(T) = \emptyset.$ Therefore $\sigma_a(T)= \Phi_1(T)  \bigsqcup \Phi_2(T)$  and $\Phi$ is a spectral a-partitioning function for $T.$ \end{proof}

Among the spectral valued functions listed in Table~1 and Table~2, we consider the following cases to illustrate  the use of Theorem \ref{thm41} and Theorem \ref{thm42}.

It is shown in \cite[Corollary 2.5 ]{RA} that if $\Psi_{aW}$ is a spectral a-partitioning function for $ T \in L(X)$, then  $\Phi_{W}$  is  a partitioning function for $T.$ When  $\Psi_{aW}$ is a spectral a-partitioning function for $ T,$ then  $ \sigma _W(T) \setminus \sigma_{SF_+^-}(T)=(E_a^0(T \setminus E^0(T))\bigsqcup (\sigma(T) \setminus \sigma_a(T)).$ Since $\Phi_{W}  \leq  \Psi_{aW},$  this result is then a direct consequence of Theorem \ref{thm42}. Moreover, combining Theorem \ref{thm42} and Corollary \ref{cor42},  we have the following theorem, characterizing the equivalence of the two  properties.

\begin{thm}\label{thm43}  Let $ T \in L(X).$ The spectral valued function  $\Psi_{aW}$ is a spectral a-partitioning function for $T$  if and only if  $\Phi_{W}$  is  a partitioning function for $T$ and  $ \sigma _W(T) \setminus \sigma_{SF_+^-}(T)=(E_a^0(T \setminus E^0(T)) \bigsqcup(\sigma(T) \setminus \sigma_a(T)).$
\end{thm}

  It is shown in \cite[Theorem 2.6]{AP} that if $\Psi_{w}$  is   a spectral a-partitioning function for $ T \in L(X)$, then  $\Phi_{B}$  is   a spectral  partitioning function for $T.$ When $\Psi_{w}$  is   a spectral a-partitioning function for $ T,  $ then  $\sigma _W(T) \setminus \sigma_{SF_+^-}(T)=(E^0(T) \setminus \Pi^0(T)) \bigsqcup (\sigma(T) \setminus \sigma_a(T)).$ As Since $\Phi_{B}  \leq  \Psi_{w},$ this result is then a direct consequence of Theorem \ref{thm41}. Moreover, as in Theorem \ref{thm43}, we have the following theorem characterizing the equivalence of the two  properties.

    \begin{thm} \label{thm44} Let $ T \in L(X).$ The spectral valued function  $\Psi_{w}$ is a spectral a-partitioning function for $T$  if and only if  $\Phi_{B}$  is  a partitioning function for $T$ and $ \sigma _W(T) \setminus \sigma_{SF_+^-}(T)=(E^0(T) \setminus \Pi^0(T)) \cup (\sigma(T) \setminus \sigma_a(T)).$

    \end{thm}

\begin{ex} Let us  consider the operator $T$ of Remark \ref{rema21}, and the two spectral valued  functions defined by
$\Phi_{gaw}(T)=( \Phi_1(T), \Phi_2(T))= (\sigma_{BW}, E_a(T))$  and
 $\Phi_{Baw}(T)=(\Phi'_1(T), \Phi'_2(T))=  (\sigma_{BW}(T), E_a^0(T)).$ Then  $ \Phi_{Baw} \leq  \Phi_{gaw} $ and
$\Phi_{Baw}$  is a spectral a-partitioning function for $T.$  Since we have
$ \Phi_2(T) \setminus \Phi'_2(T) = \{0\},$ and $(\Phi'_1(T) \setminus \Phi_1(T)) \cup  (\sigma(T)\setminus \sigma_a(T))= \emptyset,$   then from Corollary \ref{cor41}, $\Phi_{gaw}$ is not a spectral partitioning function for $T,$ which is readily verified.

\end{ex}

For the study of the spectral valued functions, we have considered  comparable spectral valued functions for the order relation $\leq.$ This not always the case, as seen by the  spectral valued functions  $\Psi_{gb}$ and $\Phi_{gab},$ defined by $\Psi_{gb}(T)=(\sigma_{SBF_+^-}(T),\Pi(T))$ and $\Phi_{gab}(T)=(\sigma_{BW}(T),\Pi_a(T)),$ for all $T \in L(X).$ We observe that  In fact $ \Psi_{gb}$ and $ \Phi_{gab}$  are comparable for  the order relation $<<$, in the sense that $\sigma_{SBF_+^-}(T) \subset \sigma_{BW}(T),$ and $\Pi(T) \subset \Pi_a(T).$ To deal with such  cases, we have the following two results

\begin{thm}\label{thm45} Let $T \in L(X)$ and let $\Phi$ be a spectral partitioning  function for $T.$ If $\Psi$ is a spectral valued function such that $\Psi <<  \Phi. $   Then $\Psi$ is a spectral a-partitioning function for  $T$ if and only if $(\Phi_1(T) \setminus \Psi_1(T)) \bigsqcup (\Phi_2(T) \setminus \Psi_2(T))= \sigma(T)\setminus \sigma_a(T).$
\end{thm}

\begin{proof} Since $\Phi$ is  a spectral partitioning  function for $T,$    then  $\sigma(T)= \Phi_1(T) \bigsqcup \Phi_2(T).$  If $\Psi$ is a spectral a-partitioning function for  $T,$  then $ \sigma(T)=  \Phi_1(T) \bigsqcup \Phi_2(T)=  [\Psi_1(T) \cup \Psi_2(T)] \bigsqcup [(\sigma(T)\setminus \sigma_a(T))]= \Psi_1(T) \cup   (\Phi_1(T) \setminus \Psi_1(T)) \cup \Psi_2(T) \cup (\Phi_2(T) \setminus \Psi_2(T)).$ As  $ (\Phi_1(T) \setminus \Psi_1(T)) \cap (\Phi_2(T) \setminus \Psi_2(T)= \emptyset,$ then $ \sigma(T)\setminus \sigma_a(T)= (\Phi_1(T) \setminus \Psi_1(T)) \bigsqcup (\Phi_2(T) \setminus \Psi_2(T).$\\
Conversely assume that $ \sigma(T)\setminus \sigma_a(T)= (\Phi_1(T) \setminus \Psi_1(T)) \bigsqcup (\Psi_2(T) \setminus \Psi_2(T)).$  As  $\sigma(T)= \Phi_1(T) \bigsqcup \Phi_2(T),$  then $\sigma_a(T)= \Psi_1(T) \cup \Psi_2(T).$ As we have obviously $ \Psi_1(T) \cap \Psi_2(T)= \emptyset, $  then $\sigma_a(T)= \Psi_1(T) \bigsqcup \Psi_2(T),$  and $\Psi$ is a spectral a-partitioning function for  $T.$ \end{proof}

Similarly to Theorem \ref{thm45}, we have the following result, which we give without proof.

\begin{thm}\label{thm46} Let $T \in L(X)$ and let $\Psi$ be a spectral a-partitioning  function for $T.$ If $\Phi$ is a spectral valued function such that $\Psi <<  \Phi.$   Then $\Phi$ is a spectral partitioning function for  $T$ if and only if $ \sigma(T)\setminus \sigma_a(T)= (\Phi_1(T) \setminus\Psi_1(T)) \bigsqcup (\Phi_2(T) \setminus \Psi_2(T)).$
\end{thm}

We observe that in Theorem \ref{thm46}, $\Phi$ is a spectral partitioning function for  $T$ if and only if the function $\Phi \setminus \Psi$ defined by $(\Phi \setminus  \Psi)(T) =  ((\Phi_1(T) \setminus\Psi_1(T)) , (\Phi_2(T) \setminus \Psi_2(T)) , \, \forall T \in L(X)$  is partitioning for the complement of the approximate spectrum $\sigma(T)\setminus \sigma_a(T).$

\goodbreak

 {\footnotesize
\noindent Mohammed Berkani,\\
\noindent Department of mathematics,\\
\noindent Science Faculty of Oujda,\\
\noindent University Mohammed I,\\
\noindent Operator Theory Team, SFO,\\
\noindent Morocco\\
\noindent berkanimo@aim.com\\}

\begin{thebibliography}{20}

\bibitem {AP} P. Aiena, P. Pe$\tilde{n}$a,   \emph{ Variations on
Weyl's theorem},  J. Math. Anal. Appl. \textbf{324} (2006) 566--579.


\bibitem {AZ} M. Amouch, H. Zguitti,  \emph{On the equivalence of
Browder's and generalized Browder's theorem},  Glasgow Math. J. {\bf
48} (2006), 179--185.

\bibitem{BAR} B. A. Barnes,  \emph{Riesz points and Weyl's theorem},  Integr. Equ. and Oper. Theory
{\bf 34} (1999), 187--196.



\bibitem {BE1} M. Berkani,  \emph{On a class of quasi-Fredholm operators},
  Integr. Equ. and Oper. Theory  {\bf 34} (1999), no. 2, p. 244--249.



\bibitem {BW} M. Berkani,  \emph{B-Weyl spectrum and poles of the
resolvent},  J. Math. Anal. Applications, {\bf 272} (2), 596--603
(2002).


\bibitem {BI} M. Berkani, \emph{ Index of B-Fredholm operators
and generalization of a Weyl theorem }, Proc. Amer. Math. Soc. {\bf 130} (2002) 1717--1723


\bibitem {B1} M. Berkani,  \emph{ On the equivalence of Weyl and generalized Weyl theorem
},  Acta Mathematica Sinica, English series, Vol. {\bf23}(no.1) (2007), 103--110.
(2002).





\bibitem{BK} M. Berkani, J.J. Koliha, \emph{Weyl type theorems for bounded
linear operators}, Acta Sci. Math. (Szeged) \textbf{69} (2003),
359--376.


 \bibitem {BS} M. Berkani, M. Sarih,  \emph{On semi B-Fredholm operators},
 Glasgow Math. J. {\bf 43} (2001), 457--465.






\bibitem {BZ1} M. Berkani, H. Zariouh,  \emph{Extended Weyl type
 theorems},  Math. Bohemica Vol. {\bf 134}, No. 4, pp. 369--378, (2009).


\bibitem{BZ3} M. Berkani, H. Zariouh,  \emph{New extended Weyl type theorems},
 Mat. Vesnik Vol. {\bf 62} (2) (2010), 145--154.

\bibitem{CM} X. Cao, B. Meng,   \emph{Browder's theorem and generalized Weyl's theorem}
Preprint: www.­math.­pku.­edu.­cn:­8000/­var/­preprint/­362.­pdf

\bibitem{CH}  R. E. Curto; Y.M. Han, \emph {Generalized Browder's and Weyl's theorems for Banach space operators,} J. Math. Anal. Appl. 336 (2007), no. 2, 1424--1442.
\bibitem{DD}{ D. S. Djordjevic,}   \emph{ \em Operators obeying a-Weyl's theorem,}
Publ. Math. Debrecen {\bf 55} (1999), no 3-4, 283-298.

\bibitem {DH} S. V. Djordjevi$\acute{c}$ and Y. M. Han, \emph{Browder's
 theorems and spectral continuity} , Glasgow Math. J. {\bf 42}  (2000),
 479--486.
\bibitem{DU} B. P. Duggal, \emph{ Polaroid operators and generalized Browder-Weyl theorems,} Math. Proc. R. Ir. Acad. {\bf 108}, (2008), no. 2, 149--164.
\bibitem {H} H. Heuser,  \emph{Functional Analysis},  John Wiley \& Sons Inc ,
New York, (1982).









\bibitem {RA} V. Rako$\check{c}$evi$\acute{c}$,  \emph{Operators
obeying a-Weyl's theorem},  Rev. Roumaine Math. Pures Appl. {\bf 34}
(1989), 915--919.

\bibitem{WL}{ Weyl, H.}  \,\,\emph { \"Uber beschr\"ankte  quadratische  Formen, deren
Differenz vollstetig ist.} Rend. Circ. Mat. Palermo 27 (1909),
373-392.


\bibitem{XH} Xiao Hong Cao: \emph{ A-Browder's Theorem and Generalized a-Weyl's Theorem}, Acta Mathematica
Sinica,English Series,Vol.23,No.5(2007), 951-960.



\end{thebibliography}
\end{document}